\documentclass[12pt]{amsart}
\usepackage{amssymb,amsfonts,amsmath,amsthm,hyperref,xcolor} 
\usepackage[letterpaper, portrait, margin=28mm]{geometry}
\usepackage{bbm}
\usepackage{comment}

\newcommand{\bs}{\backslash}

\newcommand{\DBSC}{\operatorname{DBSC}}

\newcommand{\N}{{\mathbb N}}

\newcommand{\cB}{{\mathcal B}}

\newcounter{dummy} \numberwithin{dummy}{section}
\newtheorem{theorem}[dummy]{Theorem}
\newtheorem{lemma}[dummy]{Lemma}

\newtheorem{definition}[dummy]{Definition}
\newtheorem{proposition}[dummy]{Proposition}

\theoremstyle{remark}
\newtheorem{remark}[dummy]{Remark}

\title{On the sequential monotone closure of $CD_{\omega}(K)$ spaces}

\author[S.~Chalana]{Sukrit Chalana}
\address{Department of Mathematics, Toronto Metropolitan University, 350 Victoria Street, Toronto, Canada M5B 2K3}
\email{schalana@torontomu.ca}

\author[D.~Leung]{Denny H.~Leung}
\address{Department of Mathematics, National University of Singapore, Singapore 117543}
\email{dennyhl@u.nus.edu}

\author[F.~Xanthos]{Foivos Xanthos}
\address{Department of Mathematics, Toronto Metropolitan University, 350 Victoria Street, Toronto, Canada M5B 2K3}
\email{foivos@torontomu.ca}

\date{\today}
\begin{document}

\begin{abstract}
In this short note, we settle a problem posed by Wickstead in ~\cite{W:24}, arising from the study of the Riesz completion of spaces of regular operators between Banach lattices. 
\end{abstract}


\maketitle

\section{Introduction}

Let $E$ be an Archimedean vector lattice and $E^{\delta}$ be an order completion of $E$. In the following, we will view $E$ as an order dense and majorizing sublattice of $E^{\delta}$. We put $u(E)$ to be the set of $\sigma$-upper elements of $E$ in $E^{\delta}$, that is $u(E):=\{x \in E^{\delta} \,\, | \,\, x_n \uparrow x \text{ for some } (x_n) \subset E\}$ and $l(E):=\{x \in E^{\delta} \,\, | \,\, x_n \downarrow x \text{ for some } (x_n) \subset E\}$ the $\sigma$-lower elements of $E$ in $E^{\delta}$. It is clear that $u(E)-u(E)$ is a subspace of $E^{\delta}$. Moreover, by applying the lattice identity $(x-y)^+=x \vee y-y$, one readily verifies that  this space
is in fact a sublattice of $E^{\delta}$. Moreover, it is easy to see that $u(E)-u(E)=u(E)\cup l(E)-u(E)\cup l(E)$. This justifies the  definition of the following space, which is unique up to lattice isomorphism.

\begin{definition} Let $E$ be an Archimedean vector lattice with order completion $E^{\delta}$. The sublattice $E^{\sigma m}:= u(E)-u(E)$ of $E^{\delta}$ is called the \emph{sequential monotone closure} of $E$.
\end{definition}

The above notion was, to the best of our knowledge, first studied in \cite{Q:75}. In the recent paper \cite{W:24}, Wickstead further investigated this space 
under the name of \emph{sequential order closure} and denoting it by $E^{\sigma}$. We believe that the notation $E^{\sigma m}$  and the term \emph{sequential monotone closure} are more appropriate (see Remark~\ref{remark}), and for this reason we adopt this terminology and notation throughout the paper. If $E$ is a Banach lattice, then $E^{\sigma m}$ admits the \emph{maximal extension norm} $||x||_{E^{\sigma m}}:=\inf\{||y|| \,\, |\,\, y \in E, y \geq |x|\}$ which extends the lattice norm on $E$ to a lattice norm on $E^{\sigma m}$. In \cite[Theorem 5.6]{W:24} it is proved that if $E$ is  an almost $\sigma$-order complete Banach lattice then $(E^{\sigma m},||\cdot||_{E^{\sigma m}})$ is complete and in \cite[Question 5.8]{W:24} it is asked whether $(E^{\sigma m},||\cdot||_{E^{\sigma m}})$ is complete for any Banach lattice $E$. This question is motivated by the study of the Riesz completion of spaces of regular operators. In this note, we answer this question to the negative by presenting a class of Banach lattices $E$ for which $E^{\sigma m}$ is not uniformly complete, and thus $(E^{\sigma m},||\cdot||_{E^{\sigma m}})$ is not a complete space.

\section{The main result}

We recall below some standard  terminology from vector lattice theory. Let $E,F$ be vector lattices. A linear operator $T:E \rightarrow F$ is said to be a \emph{lattice  homomorphism} if $T(x \vee y)=T(x)\vee T(y)$ for all $x,y \in E$. If, in addition to being a lattice homomorphism, $T$ is also bijection, we say that $T$ is a \emph{lattice  isomorphism}. In this case, we say that $E,F$ are lattice isomorphic and we write $E \cong F$. Let $Y$ be a sublattice of $E$, we say that $Y$ is \emph{majorizing} in $E$, if for every $x \in E$ there exists $y \in Y$ such that $x \leq y$, we say that $Y$ is \emph{order dense} in $E$, if for every $x \in E_+, x\neq 0$ there exists $y \in Y$ such that $0<y\leq x$.  We say that $E$ is {\em ($\sigma$-) order complete} if every nonempty (countable) order bounded above subset of $E$, has a supremum in $E$. Every Archimedean vector lattice is lattice isomorphic to an order dense and majorizing sublattice of an order complete space, this space is unique up to lattice isomorphism and we denote it by $E^{\delta}$ (see e.g. \cite[Section 32]{LZ:71}). Let $K$ be a nonempty Hausdorff topological space. We denote by $B(K)$ the space of bounded functions on $K$ and by $C_b(K)$ the space of bounded continuous functions.  We note that $B(K)$ is an order complete vector lattice under the pointwise ordering and a Banach lattice under the uniform norm $||\cdot||_{\infty}$. We denote by $\mathbbm{1}_{F}$  the indicator function of the set $F \subset K$. 
The set of differences of bounded semicontinuous functions is denoted by  $DBSC(K)$, that is 

$$\DBSC(K)=\{ f_1-f_2 \,\, | \,\, f_1,f_2 \,\, \text{semicontinuous and bounded} \,\, on \,\, K\}.$$

The following space was introduced in \cite{AW:93}. 

\[
CD_\omega(K) = \bigl\{ f \in B(K) \;\big|\; \exists g \in C_b(K) : [f \neq g] \text{ is at most countable} \bigr\}.
\]

The following result contains some basic facts about $CD_\omega(K)$ spaces. The proof is essentially  given in \cite{AW:93}, but the authors had the additional assumption that $K$ is compact and basically disconnected, for this reason we reproduce the proof here.
A set $A \subset K$ is called \emph{nowhere dense} if $\operatorname{int}(\overline{A}) = \emptyset$. A set $M \subset K$ is called \emph{meager} (or of first category) if it is a countable union of nowhere dense sets, i.e. $M = \bigcup_{n=1}^{\infty} N_n$ with each $N_n$ nowhere dense in $K$. We say that $K$ is \emph{perfect} if $K$ has no isolated 
points and a \emph{Baire space} if the complement of every meager set is dense in $K$. Finally, we say that $K$ is a \emph{Polish  space}, if it is a  separable and 
completely metrizable space. We also recall that every nonempty prefect  Polish space is uncountable (see e.g.~\cite[Theorem 6.2]{K:94}).

\begin{proposition}\label{propbasic}
Let $K$ be a perfect Baire space, then the space $CD_\omega(K)$ is an order dense, majorizing and closed sublattice of $(B(K),||\cdot||_\infty)$. Moreover, there exists a bounded linear projection $P_c: CD_\omega(K) \rightarrow C_b(K)$ that is also a lattice homomorphism  such that $[P_c(f) \neq f]$ is at most countable for all $f \in CD_\omega(K)$.
\end{proposition}

\begin{proof}
For $f_i,g_i \in B(K)$ and $\lambda,\mu \in \mathbb{R}$,
\[[\lambda f_1+\mu f_2 \neq \lambda g_1+\mu g_2], [f_1\vee f_2 \neq g_1 \vee g_2] \subset [f_1 \neq g_1] \cup [f_2 \neq g_2].\]
It follows immediately that $CD_\omega(K)$ is a sublattice of $B(K)$. 
Since $\mathbbm{1}_{\{x\}} \in CD_\omega(K)$ for all $x \in K$ and $\mathbbm{1}_{K} \in CD_\omega(K)$ we get that $CD_\omega(K)$ is an order dense and majorizing sublattice of $B(K)$.
 Let $f \in CD_\omega(K)$.
We claim that there is a unique $g\in C_b(K)$ so that $[f\neq g]$ is at most countable.
Existence of such a $g$ follows from the definition of $CD_\omega(K)$.  
Assume that $g,g'\in C_b(K)$ and that $[f\neq g], [f\neq g']$ are at most countable.
Since $K$ has no isolated points, $[f \neq g] \cup [f \neq g']$ is meager.
As $K$ is a Baire space, 
\[ A:= [f=g]\cap [f=g'] = ([f \neq g] \cup [f \neq g'])^c\]
 is dense in $K$.  Now $g, g'$ are two continuous functions that agree on the dense set $A$; thus $g=g'$.  This verifies the claim.
 
Define $P_c:CD_\omega(K)\to C_b(K)$ by $P_c(f) =g$, where $g$ is the unique element in $C_b(K)$ such that $[f\neq g]$ is at most countable.
Using the uniqueness of $g$ from the claim, it is easy to see that $P_c$ is a linear lattice homomorphism.  Clearly, $P_c$ is a projection.
If $f \in CD_\omega(K)$, then $|f| \leq ||f|| \mathbbm{1}_{K}$.  Hence $|P_c(f)| \leq ||f|| \mathbbm{1}_{K}$ since $P_c$ is a lattice homomorphism.  Therefore, $||P_c(f)|| \leq ||f||$ and so $P_c$ is a bounded operator.

Finally, we will establish that $CD_\omega(K)$ is closed in $(B(K),||\cdot||_{\infty})$. Let $(f_n)$ be a Cauchy sequence in $CD_{\omega}(K)$ and $f$ its uniform limit in $B(K)$, then since $P_c$ is bounded, 
$(P_c(f_n))$ is Cauchy in $C_b(K)$ and thus has a uniform limit $g \in C_b(K)$. Note that $[f \neq g] \subset \bigcup_{n=1}^{\infty} [P_c(f_n) \neq f_n]$ and therefore $[f \neq g]$ is at most countable.
So $f \in  CD_\omega(K)$, as required.
\end{proof}

Proposition \ref{propbasic} shows in particular that $B(K)$ is an order completion of $CD_{\omega}(K)$.  Thus $CD_\omega(K)^{\sigma m}$ can be realized as a sublattice of $B(K)$. The following 
lemma will be used to calculate $CD_\omega(K)^{\sigma m}$.

\begin{lemma}\label{mainlemma}
Let $E$ be an order dense sublattice of $B(K)$, and let $L$ be the set of all bounded functions on $K$ that can be expressed as the difference of pointwise increasing limits of sequences from $E$. If $f,g \in B(K)$ satisfy $g \in L$ and the set $[f \neq g]$ is at most countable, then $f \in L$.

\end{lemma}

\begin{proof}
Note first that since $E$ is order dense in $B(K)$ we have  $h\cdot \mathbbm{1}_{A} \in E$ for all finite $A \subset K$ and $h \in B(K)$. 
Let $h = f-g$.  Then $h \in B(K)$ and $[h\neq0]$ is at most countable.
Choose finite sets $F_n, G_n\subset K$, $n\in \N$, so that $(F_n)$ and $(G_n)$ increase to $\{h<0\}$ and $\{h>0\}$ respectively.
Define 
\[ d_n=h\cdot \mathbbm{1}_{F_n}, u_n=h\cdot \mathbbm{1}_{G_n}.\]
Then $(d_n)$ is a decreasing sequence and $(g_n)$ is an increasing sequence in $E$.  If $h(x) < 0$, then there exists $n_0$ such that $x\in F_n$ for all $n\geq n_0$.  Hence
\[ (h\wedge 0)(x) = h(x) = \lim d_n(x).\]
On the other hand, if $h(x) \geq 0$, then $x\notin F_n$ for all $n$ and thus 
\[ (h\wedge 0)(x) = 0 = \lim d_n(x).\]
This proves that $d_n\downarrow h\wedge0$.  Similarly, $u_n \uparrow h \vee0$.
By definition, this shows that $h\wedge 0$ and $h\vee 0$ belong to $L$.
Since $L$ is a subspace of $B(K)$, $h = h\wedge 0 + h\vee 0 \in L$.
Finally, $f = g+h \in L$ since $g,h \in L$ and $L$ is a subspace of $B(K).$
\end{proof}

Below we will establish a connection between $\DBSC(K)$ and $CD_\omega(K)^{\sigma m}$. We recall here that if $K$ is metrizable, then a function $f \in B(K)$ is lower  (upper) semicontinuous, if and only if there exists a sequence $(f_n) \subset C(K)$ such that $f_n \uparrow f$ ($f_n \downarrow f$) (see e.g. \cite[Lemma 2.41, Theorem 3.13]{AB:06}). 

\begin{theorem}\label{thm1} Let $K$ be a perfect metrizable Baire space, then
$$CD_\omega(K)^{\sigma m}= \{f \in B(K) \,\, |\,\, \exists  g \in \DBSC(K) \,\, : [f \neq g] \text{ is at most countable} \}$$
Moreover, we have $E^{ \sigma m}=CD_\omega(K)^{\sigma m}$ for any order dense sublattice $E$ of $B(K)$ such that $C_b(K) \subset E \subset CD_\omega(K)$.
\end{theorem}

\begin{proof}
We put $L:=\{f \in B(K) \,\, |\,\, \exists  g \in \DBSC(K) \,\, : [f \neq g] \text{ is at most countable}\}$. By Proposition~\ref{propbasic} it follows that $B(K)$ is an order completion of $CD_\omega(K)$ and thus $CD_\omega(K)^{\sigma m}=u(CD_\omega(K))-u(CD_\omega(K))$. Let $f\in u(CD_\omega(K))$.  Then there exists $(f_n)\subset CD_\omega(K)$ such that $f_n \uparrow f$. The operator $P_c$ is positive by Proposition~\ref{propbasic} and hence  $(P_c(f_n))$ is increasing and order bounded in $C_b(K)$. Therefore $(P_c(f_n))$ increases pointwise to a  bounded lower semicontinuous function $g$. Since $[f \neq g] \subset \bigcup_{n=1}^\infty [P_c(f_n) \neq f_n]$ and $[P_c(f_n)\neq f_n]$ is at most countable for each $n$ by Proposition~\ref{propbasic},  $[f \neq g]$ is at most countable. If $h\in  u(CD_\omega(K))-u(CD_\omega(K))$,
write $h = h_1-h_2$ with $h_i \in u(CD_\omega(K))$, $i =1,2$. By the preceding argument, there are bounded lower semicontinuous functions $g_1,g_2$ so that $[h_i\neq g_i]$ is at most countable for $i=1,2$. Then $g: = g_1-g_2\in \DBSC(K)$ and $[h\neq g]\subseteq [h_1\neq g_1] \cup [h_2\neq g_2]$ is at most countable.  This proves that $ u(CD_\omega(K))-u(CD_\omega(K)) \subset L$.

Conversely, let $f\in B(K)$ be such that there exists $g \in \DBSC(K)$ such that $[f \neq g]$ is at most countable.  It follows from the definitions that  $\DBSC(K)\subset u(CD_\omega(K))- u(CD_\omega(K))$.  
Since $CD_\omega(K)$ is order dense in $B(K)$ by Proposition \ref{propbasic},  we may apply  Lemma~\ref{mainlemma} to see that $f\in  u(CD_\omega(K))-u(CD_\omega(K))$.

Finally, let $E$ be an order dense sublattice  of $B(K)$ such that $C_b(K) \subset E \subset CD_\omega(K)$. Then $B(K)$ is an order completion of $E$, and thus $E^{ \sigma m}=u(E)-u(E)$. Since $E \subset CD_\omega(K)$, we clearly have $E^{ \sigma m} \subset CD_\omega(K)^{\sigma m}$. On the other hand, $C_b(K)\subset E$ implies that $\DBSC(K) \subset u(E)-u(E)$.   
Hence 
\[ CD_\omega(K)^{\sigma m} = L \subset E^{\sigma m}\]
by Lemma~\ref{mainlemma}.
\end{proof}

We turn now our focus on the uniform limits of sequences in $\DBSC(K)$. We denote by $\cB^1_1(K)$ the closure of $\DBSC(K)$ in $(B(K),|| \cdot||_\infty)$.  Assume that there is a sequence $(F_n)^\infty_{n=0}$ of nonempty closed subsets of $K$ so that 
\begin{enumerate}
\item[(a)] $F_0=K$, $F_n \subsetneq F_{n-1}$ for any $n\in \N$.
\item[(b)] If $x\in F_n$, $n\in \N$, and $U$ is an open neighbourhood of $x$, then $U\cap(F_{n-1}\bs F_n)$ is uncountable.
\end{enumerate}
Define $f = \sum^\infty_{n=1}\frac{(-1)^n}{n}\mathbbm{1}_{F_{n-1}\bs F_n}$.  
Since $\mathbbm{1}_{F_{n-1}\bs F_n} = \mathbbm{1}_{F_{n-1}} - \mathbbm{1}_{F_n}$ is a difference of bounded upper semicontinuous functions and the series for $f$ converges uniformly on $K$, $f\in \cB^1_1(K)$.

\begin{proposition}\label{p1}
Let $K$ be a metrizable space that contains a  sequence $(F_n)^\infty_{n=0}$ of nonempty closed sets that satisfy (a) and (b). Define $f$ as above. If $g\in \DBSC(K)$, then $S:=[f\neq g]$ is uncountable.
\end{proposition}

\begin{proof}
Suppose on the contrary that there exists $g\in \DBSC(K)$ so that $S: = [f\neq g]$ is at most countable. We may write $g$ as $h-k$, where $h,k$ are bounded lower semicontinuous functions on $K$.
Since $f = \frac{(-1)^n}{n}$ on $F_{n-1}\bs F_n$, $h = k + \frac{(-1)^n}{n}$ on $F_{n-1}\bs (F_n \cup S)$ for any $n\in \N$.  For any $n$,  $F_n\neq \emptyset$; hence it follows from (b) that $F_{n-1}\bs (F_n\cup S)$ is nonempty.
Fix $n\in \N$ and suppose that $x\in F_{2n-1}\bs (F_{2n}\cup S)$. Let $B(x,\frac{1}{m})$ be the open unit ball with radius $\frac{1}{m}$, centered at $x$. By (b), we have that $B(x,\frac{1}{m}) \cap (F_{2n-2}\bs F_{2n-1})$ is uncountable. Thus $B(x,\frac{1}{m}) \cap F_{2n-2}\bs(F_{2n-1} \cup S) \neq \emptyset$ for each $m \in \mathbb{N}$. In particular, there exists a sequence  $(x_m)$ in $F_{2n-2}\bs(F_{2n-1}\cup S)$ that converges to $x$. By lower semicontinuity of $h$,
\begin{align*}
k(x) +\frac{1}{2n} & = 
h(x) \leq \liminf h(x_m)  \\&= \liminf k(x_m) -\frac{1}{2n-1}\leq \sup_{y\in F_{2n-2}\bs (F_{2n-1}\cup S)}k(y) - \frac{1}{2n-1}.\notag\end{align*}
Hence 
\begin{equation}\label{eq1}
\sup_{x\in F_{2n-1}\bs(F_{2n}\cup S)}k(x) \leq \sup_{y\in F_{2n-2}\bs (F_{2n-1}\cup S)}k(y) - \frac{1}{2n-1}-\frac{1}{2n}.
\end{equation}
Now assume that $z\in  F_{2n}\bs (F_{2n+1}\cup S)$. By (b) again, we can choose a sequence $(y_m)$ in $F_{2n-1}\bs (F_{2n}\cup S)$ that converges to $z$.
By (\ref{eq1}) and lower semicontinuity of $k$,
\[ k(z)  \leq \liminf k(y_m) \leq\sup_{y\in F_{2n-2}\bs (F_{2n-1}\cup S)}k(y) - \frac{1}{2n-1}-\frac{1}{2n}.
\]
It follows that 
\begin{equation}\label{eq2} \sup_{z\in F_{2n}\bs (F_{2n+1}\cup S)}k(z) \leq \sup_{y\in F_{2n-2}\bs (F_{2n-1}\cup S)}k(y)-\frac{1}{2n-1} - \frac{1}{2n}\ \text{ for any $n\in \N$}.
\end{equation}
Hence
\[ \sup_{z\in F_{2n}\bs(F_{2n+1}\cup S)}k(z) \leq \sup_{y\in F_0\bs (F_1\cup S)}k(y)-\sum^{2n}_{j=1}\frac{1}{j}.\]
This is impossible since $k$ is bounded and $ F_{2n}\bs(F_{2n+1}\cup S)$ is nonempty for any $n$.
\end{proof}

\begin{theorem}\label{thm2}
Let $K$ be an uncountable Polish space.  There exists $f\in \cB^1_1(K)$ so that $[f\neq g]$ is uncountable for any $g\in \DBSC(K)$.
\end{theorem}

\begin{proof}
In view of Proposition \ref{p1}, it suffices to construct a sequence of closed sets $(F_n)^\infty_{n=0}$ 
in $K$ satisfying conditions (a) and (b).
By \cite[Corollary 6.5]{K:94}, $K$ contains a homeomorphic copy of the Cantor set, which we realize as $\{0,1\}^{\N\times \N}$.
Let $\iota:\{0,1\}^{\N\times \N}\to K$ be a homeomorphic embedding.
For $n\in \N$, define 
\[ A_n = \{a\in \{0,1\}^{\N\times \N}: a(i,j) = 0 \text{ if $1\leq i\leq n$}\} \text{ and } F_n = \iota(A_n).\] 
Also set $F_0 = K$. Clearly $(F_n)^\infty_{n=0}$ is a sequence of closed sets in $K$ so that $F_n\subsetneq F_{n-1}$ for any $n\in \N$.
Suppose that $x\in F_n$ for some $n\in \N$.  Thus $x= \iota(a)$ for some $a\in A_n$.
Let $U$ be an open neighbourhood of $x$. Then $V:=\iota^{-1}(U)$ is an open neighbourhood of $a$ in $\{0,1\}^{\N\times \N}$.  
There is a finite set $I\subseteq \N\times \N$ so that 
\[ W: = \{b\in \{0,1\}^{\N\times \N}: b(i,j) = a(i,j) \ \forall (i,j)\in I\} \subseteq V.\]
Pick $(n,j_0)\in I^c$.
If $b\in \{0,1\}^{\N\times \N}$ satisfies $b(n,j_0) =1$ and $b(i,j) = a(i,j)$  whenever $(i,j) \in I \cup (\{1,\dots,n-1\}\times \N)$, 
then $b\in W$ and hence $\iota(b) \in U$.  
Also, $b(i,j) = a(i,j) =0$ if $1\leq i\leq n-1$.  Thus $b\in A_{n-1}$.  Since $b(n,j_0) \neq0$, $b\notin A_n$. 
Therefore, $\iota(b) \in F_{n-1}\bs F_n$.
Since there are uncountably many such $b$'s, $U\cap (F_{n-1}\bs F_n)$ is uncountable. 
\end{proof}

\begin{remark}\label{remark}
We recall that a sequence $(x_n)$ in an order complete vector lattice $E$ order converges to $x \in E$ if there exists a sequence $(y_n)$ in $E$ such that $|x_n-x| \leq y_n \downarrow 0$. Every $x \in E^{\sigma m}$ is clearly an order limit of a sequence in $E$. In the following, we show that for a class of $CD_{\omega}(K)$-spaces there exist sequential order limits from $CD_{\omega}(K)$ that do  not belong to $CD_{\omega}(K)^{\sigma m}$. First note that order convergence for order bounded sequences in $B(K)$ is equivalent to pointwise convergence (apply e.g. \cite[Lemma 8.17]{AB:06}, \cite[Corollary 2.12]{GTX:17}). Next, take $K$ to be a perfect Polish space. By Theorem~\ref{thm1}, we may identify $CD_{\omega}(K)^{\sigma m}$ with $\{f \in B(K) \,\, |\,\, \exists  g \in \DBSC(K) \,\, : [f \neq g] \text{ is at most countable} \}$. Take  $f$ as in Theorem~\ref{thm2}. Then by applying \cite[Lemma 24.11]{K:94}, one can readily see that $f$ is the pointwise limit of a sequence $(g_n)$ from $CD_\omega(K)$ and by replacing $(g_n)$ with $g_n \wedge ||f||_{\infty} \vee (-||f||_{\infty})$ we may assume that $(g_n)$ is order bounded. In particular, $f$ is an order limit of a sequence from  $CD_{\omega}(K)$, yet  Theorem~\ref{thm2}, asserts that $f$ is not in $CD_{\omega}(K)^{\sigma m}$.
\end{remark}

We are now in position to state our main result. Let $E$ be an Archimedean vector lattice and $e \in E_+$. The principal ideal of $e$ and the corresponding norm on it are defined as follows: $I_e=\{x \in E \,\, | \,\, |x| \leq \lambda e \text{ for some } \lambda \in \mathbb{R}_+\},||x||_e=\inf\{\lambda \geq 0 \,\, | \,\, |x| \leq \lambda e\}$. We say that $E$ is \emph{uniformly complete} if $(I_e,||\cdot||_e)$ is complete for all $e \in E_+$. We also note that every Banach lattice is uniformly complete (see e.g.~ \cite[Corollary 2.59]{AT:07}), and that uniform completeness is invariant between lattice isomorphic spaces. 

\begin{theorem}\label{mainthm}
Let $K$ be a perfect Polish space, then $CD_{\omega}(K)^{\sigma m}$ is not uniformly complete.
\end{theorem}

\begin{proof}
In view of Theorem~\ref{thm1} we may identify $CD_{\omega}(K)^{\sigma m}$ with 
$\{f \in B(K) \,\, |\,\, \exists  g \in \DBSC(K) \,\, : [f \neq g] \text{ is at most countable} \}$. Theorem~\ref{thm2} asserts that $(CD_{\omega}(K)^{\sigma m},||\cdot||_{\mathbbm{1}_{K}})$  is not complete and the claim follows.

 \end{proof}

We recall that $E$ is said to be an \emph{almost $\sigma$-order complete} vector lattice if it can be embedded as a super order dense sublattice of a  $\sigma$-order complete vector lattice (see e.g. \cite{AL:74,Q:75}). Next, we examine the relationship between the $\sigma$-order completeness of $E^{\sigma m}$ and the almost $\sigma$-order completeness of $E$. The following result is well known, for the convenience of the reader, we provide an independent  proof here.

\begin{proposition}
Let $E$ be an almost $\sigma$-order complete vector lattice, then $E^{\sigma m}$ is $\sigma$-order complete. 
\end{proposition}

\begin{proof}
Let $X$ be $\sigma$-order complete space such that $E$ can be viewed as a super order dense sublattice of $X$, that is, $E$ is a sublattice of $X$ and  for every $x \in X_+$ there exists $(x_n) \subset E_+$ such that $x_n \uparrow x$. In the following we will also view $X$ as an order dense and majorizing sublattice of its order completion $X^{\delta}$. Let $I:=\{y \in X_{\delta} \,\, | \,\, \exists x \in E \,\, : \,\, |y|\leq |x|\}$ be the ideal generated by $E$ in $X^{\delta}$. Then $I$ is order complete and  $E$ is order dense and majorizing in $I$, thus $I\cong E^{\delta}$.  We put $u(E)$ and $l(E)$ to be the set of $\sigma$-upper elements and $\sigma$-lower elements of $E$ in $I$ respectively.  We will show next that $l(E)=u(E)$. Take $x \in u(E)$, then  $x \in X$ and there exists $v \in E$ such that $x \leq v$, in particular $v-x \in X_+$ thus there exists a sequence $(v_n) \subset E_+$ such that $v_n \uparrow v-x$, thus $v-v_n \downarrow x$ and $x \in l(E)$. Similarly, we can show $l(E) \subset u(E)$. Now, we note  $(u(E)-u(E))^+=(u(E)-l(E))^+=(u(E)+u(E))^+=u(E)^+$. Take $(x_n) \subset (E^{\sigma m})_+=u(E)^+$ such that $(x_n)$ is increasing and order bounded in $E^{\sigma m}$.  Then there exists $x \in I$ such that $x_n \uparrow x$ and $(w_{n,k})_{k \in \mathbb{N}}$ in $E_+$ such that $x_n=\displaystyle\sup_{k \in \mathbb{N}}w_{n,k}$ for all $n \in \mathbb{N}$. Put $w_n=\bigvee_{k=1}^n w_{n,k}$ for all $n \in \mathbb{N}$, then $(w_n) \subset E$ and it is easy to see $w_n \uparrow x$, thus $x \in E^{\sigma m}$ and the claim follows. 
\end{proof}

The next result shows  that the  converse of the above fact is not true in general.

\begin{proposition}
Let $K_{\infty}=K\cup\{\infty\}$ be the one point compactification of an uncountable set $K$ equipped  with the the discrete topology. Then $C(K_{\infty})^{\sigma m}$ is $\sigma$-order complete, yet $C(K_{\infty})$ is not almost $\sigma$-complete.
\end{proposition}

\begin{proof}
The fact that $C(K_{\infty})$ is not almost $\sigma$-order complete  follows by \cite[Example 3.4]{W:24}. We define $T:C(K_{\infty}) \rightarrow B(K)$ as follows $T(f)=f\restriction_K$ for all $f \in C(K_{\infty})$. It is clear that $T$ is a lattice homomorphism.  We also note that $T$ is one to one. Indeed, let $f,g \in C(K_{\infty})$ such that $f\restriction_K=g\restriction_K$. Since $f,g$ are continuous on $K_{\infty}$ and $K$ is dense in $K_{\infty}$ we get $f=g$. It is also clear that $\mathbbm{1}_{K} \in T(C(K_{\infty}))$ and $\mathbbm{1}_{x} \in T(C(K_{\infty}))$ for all $x \in K$. The above shows $C(K_{\infty}) \cong T(C(K_{\infty}))$ and $C(K_{\infty})^{\delta}\cong B(K)$. In the following we will identify $C(K_{\infty})$ with $T(C(K_{\infty}))$. We will show

$$u(C(K_{\infty}))-u(C(K_{\infty}))=\{f \in B(K) \,\, | \,\, \exists c \in \mathbb{R} \,\, : \,\, [f \neq c] \,\, \text{at most} \,\, \text{countable} \}$$

We note first  $\{f \in B(K) \,\, | \,\, \exists c \in \mathbb{R} \,\, : \,\, [f \neq c] \,\, \text{at most} \,\, \text{countable} \}$ is a $\sigma$-order complete sublattice of $B(K)$. Let $f \in B(K)$ and $c \in \mathbb{R}$ such that $[f \neq c]$ is at most countable. Then since the constant function $c$ is in  $C(K_{\infty})$,  an application of Lemma~\ref{mainlemma} yields $f \in  u(C(K_{\infty}))-u(C(K_{\infty}))$.  Next we show that $C(K_{\infty}) \subset \{f \in B(K) \,\, | \,\, \exists c \in \mathbb{R} \,\, : \,\, [f \neq c] \,\, \text{at most} \,\, \text{countable} \}$. Let $f \in C(K_{\infty})$ and $c=f(\infty)$. For each $n \in \mathbb{N}$, put $U_n=f^{-1}(c-\frac{1}{n},c+\frac{1}{n})$, then $U_n$ is open in $K_\infty$, thus $U_n^c$ is a closed subset of $K_{\infty}\setminus\{\infty\}$ and therefore $U_n^c$ is at most finite.  We also  have that $[f \neq c] \subset \bigcup_{n=1}^\infty U_n^c$, thus $[f \neq c]$ is at most countable.  Thus we have $u(C(K_{\infty}))-u(C(K_{\infty})) \subset \{f \in B(K) \,\, | \,\, \exists c \in \mathbb{R} \,\, : \,\, [f \neq c] \,\, \text{at most} \,\, \text{countable} \}$ and the claim follows. 
\end{proof}

Let $X$ be a partially ordered vector space, we say that $X^{\rho}$ is a Riesz completion of $X$ whenever $X^{\rho}$ is a vector lattice and there exists a bipositive linear map $i:X \rightarrow X^{\rho}$ such that $i(X)$ is order dense in $X^{\rho}$ and the sublattice of $X^{\rho}$ generated by $i(X)$ is equal to $X^{\rho}$. Let $E,F$ be Banach lattices, then the space of all regular operators between $E$ and $F$ is the space generated by the positive linear operators $T:E \rightarrow F$ and is denoted by $\mathcal{L}^r(E,F)$. In \cite{W:24}, Wickstead gave a sufficient condition on $F$ under which the Riesz completion of $\mathcal{L}^r(c,F)$, where $c$ is the space of convergent sequences, is uniformly complete. In the following we improve this result by providing a necessary and sufficient condition.

\begin{theorem}
Let $F$ be a Banach lattice. The following are equivalent. 

\begin{enumerate}
\item[(i)] $F^{\sigma m}$ is uniformly complete.
\item[(ii)] The  Riesz completion of $\mathcal{L}^r(c,F)$ is  uniformly complete.
\end{enumerate}

\end{theorem}

\begin{proof}
Let $F^{\delta}$ be an order completion of $F$. In the following we will view $F,F^{\sigma m}$ as order dense and majorizing sublattices of $F^{\delta}$. In view of \cite[Lemma 3.2.1]{GK:18} we may  extend the norm of $F$ to a complete lattice norm on $F^{\delta}$. Let $T \in \mathcal{L}^r(E,F^{\delta})$, we denote by $||T||$ the operator norm of $T$ and  $||T||_r=|| \,|T|\, ||$ the regular norm of $T$. The space  $\mathcal{L}^r(E,F^{\delta})$ is a Banach lattice under the regular norm (see e.g.  \cite[Theorem 4.74]{AB:06}). By \cite[Theorem 5.5]{W:24}, the Riesz completion of $\mathcal{L}^r(c,F)$ can be identified with the sublattice $\mathcal{H}:=\{ T \in \mathcal{L}^r(c,F^{\sigma m}) \,\, | \,\, T(c_0) \subset F\}$ of $\mathcal{L}^r(c,F^{\delta})$, where $c_0$ is the space of null sequences (see also \cite[Theorem 3.2]{SX:25}). 

$(i) \Rightarrow (ii):$ We will show that $\mathcal{H}$ is a closed sublattice of $(\mathcal{L}^r(E,F^{\delta}),||\cdot||_r)$.  Indeed, let $(T_n) \subset \mathcal{H}$ such that $T_n \xrightarrow{{||\cdot||_{r}}} T \in \mathcal{L}^r(c,F^{\delta})$. Fix $x \in c$, then we have $|T_n(x)-T(x)| \leq  |T_n-T|(x) \Rightarrow ||T_n(x)-T(x)|| \leq ||T_n-T||_r\cdot||x||$, thus $T_n(x) \xrightarrow{{|| \cdot ||}} T(x)$. Since $(T_n(x)) \subset F^{\sigma m}$ and $F^{\sigma m}$ is uniformly complete we get $T(x) \in F^{\sigma m}$ (apply e.g.  \cite[Lemma~1.1]{TT:20}), thus $T \in \mathcal{L}^r(E,F^{\sigma m})$. Similarly we get $T(x) \in F$ for all $x \in c_0$. Therefore $\mathcal{H}$ is closed in $(\mathcal{L}^r(E,F^{\delta}),||\cdot||_r)$ and the claim follows.

$(ii) \Rightarrow (i):$ Suppose that $F^{\sigma m}$ is not uniformly complete, then by \cite[Theorem 9.3]{AB:06} we can find a $y_0 \in (F^{\sigma m})_+$ and a positive $||\cdot||_{y_0}$-Cauchy sequence $(f_n)$ in $I_{y_0}$ that is not $||\cdot||_{y_0}$ convergent in $I_{y_0}$. Let $\phi:c \rightarrow \mathbb{R}$ be the limit functional and $\textbf{1}\in c$ the constant one sequence. Put $T_n(x)=\phi(x)f_n$ for all $n \in \mathbb{N}$ and $T_0(x)=\phi(x) y_0$, then $\{T_n\}_{n=0}^\infty \subset \mathcal{H}$. Moreover $(T_n)$ is Cauchy in $(I_{T_0},||\cdot||_{T_0})$. Indeed, let $\epsilon>0$, then there exists $n_0 \in \mathbb{N}$ such that $|f_n-f_m| \leq \epsilon y_0$ for all $n,m \geq n_0$. Fix $x \in c_+$, then for every $z \in c$ such that $|z| \in [0,x]$ we have 
$|T_n(z)-T_m(z)|=|\phi(z)||f_n-f_m|\leq \phi(x)|f_n-f_m|$. Therefore by the Riesz-Kantorovich formula we get $|T_n-T_m|(x)=\sup\{|T_n(z)-T_m(z)| \,\,|\,\, |z| \in [0,x]\}\leq \phi(x)|f_n-f_m|$. In particular, we have $|T_n-T_m|\leq \epsilon T_0$ for all $n,m \geq n_0$. Thus $(T_n)$ is Cauchy in $(I_{T_0},||\cdot||_{T_0})$ and there exists $T \in I_{T_0}$ such that $(T_n)$ converges to $T$ with respect to $||\cdot||_{T_0}$.  Now, since $|T_n(\textbf{1})-T(\textbf{1})|=|f_n-T(\textbf{1})| \leq |T_n-T|(\textbf{1})$ for all $n \in \mathbb{N}$, and $T_0(\textbf{1})=y_0$,  it follows that $f_n \xrightarrow{||\cdot||_{y_0}} T(\textbf{1})\in  I_{y_0}$, which is a contradiction. 

\end{proof}

{\footnotesize

\end{document}